\documentclass[a4paper]{article}
\usepackage{amsmath,amsthm,amssymb,enumerate,hyperref}
\usepackage[legacycolonsymbols]{mathtools}
\providecommand\given{}
\newcommand\SetSymbol[1][]{%
  \nonscript\:#1\vert
  \allowbreak
  \nonscript\:
  \mathopen{}}
\DeclarePairedDelimiterX\Set[1]{\{}{\}}{%
  \renewcommand\given{\SetSymbol[\delimsize]}
  #1}

\DeclarePairedDelimiterXPP\pospart[1]{}{(}{)}{^+}{#1}
\DeclarePairedDelimiterXPP\negpart[1]{}{(}{)}{^-}{#1}

\newcommand\R{\mathbb{R}}

\newcommand\PP{\mathcal{P}}
\newcommand\QQ{\mathcal{Q}}

\newtheorem{thm}{Theorem}[section]
\newtheorem{prop}[thm]{Proposition}
\newtheorem{cor}[thm]{Corollary}
\theoremstyle{definition}
\newtheorem{defn}[thm]{Definition}
\newtheorem{ex}[thm]{Example}

\DeclareMathOperator*{\gr}{gr}

\DeclareMathOperator*{\conv}{conv}
\DeclareMathOperator*{\cone}{cone}
\DeclareMathOperator*{\dom}{dom}
\DeclareMathOperator*{\lin}{lin}

\makeatletter
\newcommand{\leqnomode}{\tagsleft@true\let\veqno\@@leqno}
\newcommand{\reqnomode}{\tagsleft@false\let\veqno\@@eqno}
\makeatother

\title{The natural ordering cone of a polyhedral convex set-valued objective mapping}
\author{Andreas L{\"o}hne\thanks{Friedrich Schiller University Jena, Faculty of Mathematics and Computer Science, 07737 Jena, Germany, andreas.loehne@uni-jena.de}}

\begin{document}
\maketitle

\begin{abstract} \noindent For a given polyhedral convex set-valued mapping we define a polyhedral convex cone which we call the natural ordering cone. We show that the solution behavior of a polyhedral convex set optimization problem can be characterized by this cone. Under appropriate assumptions the natural ordering cone is proven to be the smallest ordering cone which makes a polyhedral convex set optimization problem solvable.    

\medskip

\end{abstract}

\section{Introduction}

Polyhedral convex set optimization is the simplest problem class in set optimization, comparable to linear programming in the field of scalar optimization. Vector linear programming and multiobjective linear programming, see e.g.\ \cite{Zeleny74,Loe11,Luc16}, are special cases. Nevertheless there are some open questions with respect to the solution behavior of polyhedral convex set optimization problems, which were pointed out in \cite{HeyLoe}: A multiple objective linear program has a solution whenever the problem is feasible and the so-called \emph{upper image} is free of lines, see e.g.\ \cite[Corollary 6]{LoeWei16}. A 
natural extension of the solution concept to the set-valued case reveals a different behavior: A solution does not necessarily exist under these assumptions, see the examples in \cite{HeyLoe}. It seems to be a consequence of this issue that solution methods exist only for the case of bounded polyhedral set optimization problems so far \cite{LoeSch13}. 

Here we provide a precise description of the differences in the solution behavior between vector linear programming and polyhedral convex set optimization. It turns out that an additional assumption is required in the set-valued case. To describe this assumption, we introduce the \emph{natural ordering cone} of a polyhedral convex set-valued mapping. In order to have the same solution behavior for the two problem classes it is required that the ordering cone is a superset of the natural ordering cone.

This paper uses the \emph{complete lattice approach} to set optimization. This means that the solution concept is based on both \emph{infimum attainment} (with respect to a suitable complete lattice) and \emph{minimality}. Note that several applications of set optimization, in particular those in mathematical finance related to set-valued measures of risk, are based on this complete lattice approach. Bibliographic notes on different approaches to set optimization and corresponding applications can be found, for instance, in \cite{HamHeyLoeRudSch15, KhaTamZal15}.

The article is structured as follows. In Section \ref{sec:prel} the problem setting is described and necessary notation is recalled. The natural ordering cone is introduced in Section \ref{sec:nat_cone}. The main result, Theorem \ref{thm2}, provides a characterization of the solution behavior of a polyhedral convex set optimization problem. Section \ref{sec:vr} is devoted to the special case of vector linear programming. We point out that the natural ordering cone is not relevant in this case. In the last section we slightly adapt the solution concept taking into account the results on the natural ordering cone. Throughout we consider a running example to illustrate the ideas.

\section{Problem setting and preliminaries}\label{sec:prel}

\noindent Let $F:\R^n\rightrightarrows \R^q$ be a \emph{polyhedral convex} set-valued mapping, that is, the graph of $F$,
$$ \gr F \coloneqq \Set{(x,y) \in \R^n\times \R^q \given y \in F(x)},$$
is a convex polyhedron. A \emph{polyhedral convex set-optimization problem} is the problem to minimize $F$ with respect to the partial ordering $\supseteq$ (inverse set inclusion). Frequently in polyhedral convex set optimization, $F$ is given together with a polyhedral convex cone $C$, called \emph{ordering cone}. Then, the goal is to minimize the set-valued mapping 
$$F_C:\R^n\rightrightarrows \R^q,\quad F_C(x) \coloneqq F(x) + C,$$
which results in the problem
\leqnomode
\begin{gather}\label{eq:p}\tag{P}
  \min F(x) + C\quad \text{ s.t. }\; x \in \R^n.
\end{gather}
\reqnomode
Let $C\subseteq \R^q$ be a nonempty polyhedral convex cone. Then 
$$ y^1 \leq_C y^2 \; \iff \; y^2 \in \Set{y^1} +  C$$
defines a reflexive and transitive ordering on $\R^q$. Problem \eqref{eq:p} can be interpreted as minimizing an arbitrary polyhedral convex set-valued mapping $F$ with respect to the reflexive and transitive ordering $\preccurlyeq_C$ on the power set of $\R^q$, defined by
$$ Y^1 \preccurlyeq_C Y^2  \;\iff\; Y^2 \subseteq Y^1 + C.$$
Of course, the original problem of minimizing $F$ without mentioning any ordering cone is re-obtained by the choice $C=\Set{0}$. 

The \emph{upper image} of \eqref{eq:p} is defined by
$$ \PP \coloneqq C + \bigcup_{x \in \R^n} F(x).$$ 
It is a convex polyhedron, for details, see e.g.\ \cite{HeyLoe}. 
The upper image can be interpreted as the infimum of $F$ over $\R^n$ with respect to the ordering $\preccurlyeq_C$, i.e., we can write $\PP = \inf_{x \in \R^n} F(x)$,
see e.g.\ \cite{HamHeyLoeRudSch15}. 

The \emph{recession mapping} of $F:\R^n\rightrightarrows \R^q$ is the set-valued mapping $G:\R^n\rightrightarrows \R^q$ whose graph is the recession cone of the graph of $F$, i.e., we have
$$ \gr G = 0^+\gr F.$$

Now we consider the homogeneous problem
\leqnomode
\begin{gather}\label{eq:q}\tag{Q}
  \min G(x)+C \quad \text{ s.t. }\; x \in \R^n.
\end{gather}
\reqnomode
Its upper image is denoted by
$$ \QQ \coloneqq C + \bigcup_{x \in \R^n} G(x).$$
It is the recession cone of $\PP$, i.e.,
$$\QQ = 0^+ \PP,$$
for details see e.g.\ \cite{HeyLoe}.
In particular, this means that $\QQ$ is a polyhedral convex cone. Moreover, $G(0)$ is a polyhedral convex cone and we have, see e.g.\ \cite{HeyLoe},
\begin{equation}\label{eq:451}
	F(x) = F(x) + G(0).
\end{equation}
Let $G_C$ denote the recession function of $F_C$. Then we have $G_C(x)=G(x)+C$ for all $x \in \R^n$, see e.g.\ \cite{HeyLoe}. 

We denote by $\lin C \coloneqq C\cap(-C)$ the lineality space of $C$.
A polyhedral convex cone $C\subseteq \R^q$ is called a \emph{regular} ordering cone for a polyhedral convex set-valued objective mapping $F:\R^n\rightrightarrows \R^q$ if
\begin{equation}\label{eq:reg}
	\lin C \subseteq G(0) \subseteq C.
\end{equation}
Regularity provides a connection between the objective function and the ordering cone. 
There are three reasons for such an assumption: First, some conditions in the main result, Theorem \ref{thm2} below, become simpler. Secondly, in view of existing applications, regularity seems to be satisfied in many cases. And finally, the main result for the general case without regularity can be obtained easily as a corollary.
The first inclusion means that if $C$ contains a line then also the common recession cone $G(0)$ of the nonempty values $F(x)$ (compare \cite[Proposition 2]{HeyLoe}) must contain this line. The second inclusion means that the common recession cone $G(0)$ of the nonempty values $F(x)$ is contained in the ordering cone $C$.

\begin{prop}\label{prop11}
	$G(0)$ is a regular ordering cone for $F$. Moreover, $G(0)+C$ is a regular ordering cone for $F_C$. 		
\end{prop} 
\begin{proof}
	Obviously, we have $\lin G(0) \subseteq G(0) \subseteq G(0)$.
	Since $G_C(0) = G(0)+C$, we have $\lin (G(0)+C) \subseteq G_C(0) \subseteq G(0) + C$.
\end{proof}
Proposition \ref{prop11} shows that $G(0)$ is the smallest regular ordering cone for $F$. Moreover, because of \eqref{eq:451}, the same problem \eqref{eq:p} is obtained if $F$ is replaced by $\bar F \coloneqq F_C$ and $C$ is replaced by $\bar C \coloneqq G(0)+C$. But, $\bar C$ is always regular for $\bar F$.

We denote by 
$$G(\R^n) \coloneqq \bigcup_{x \in \R^n} G(x)$$
the {\em image} of $G$ over $\R^n$, which is a polyhedral convex cone by the same arguments as for the upper image $\QQ$ of \eqref{eq:q}.
The cone $G(\R^n)$ is a regular ordering cone for $F$ if $\lin G(\R^n) \subseteq G(0)$ in particular, if $\lin G(\R^n) = \Set{0}$. If $\lin G(\R^n) \supsetneq \Set{0}$, $G(\R^n)$ is a regular ordering cone for $F_{\lin G(\R^n)}$.

Subsequently in this article, we consider problem \eqref{eq:p} for a polyhedral convex set-valued objective mapping $F:\R^n\rightrightarrows \R^q$ and a polyhedral convex ordering cone $C \subseteq \R^q$ which is regular for $F$. We are interested in the existence of solutions for \eqref{eq:p}. The following solution concept for \eqref{eq:p} was introduced in \cite{Loe} with precursors in \cite{HeyLoe11, LoeSch13, HeyLoe}. Here $\cone Y$ denotes the \emph{conic hull}, i.e., the smallest (with respect to $\subseteq$) convex cone containing $Y$. We set $\cone \emptyset = \Set{0}$.

A tuple $(\bar{S},\hat{S}$) of finite sets $\bar{S} \subseteq \dom F$, $\bar{S} \neq \emptyset$ and $\hat{S} \subseteq \dom G \setminus \{0\}$ is called a {\em finite infimizer} \cite{HeyLoe} for problem \eqref{eq:p} if 
\begin{equation}{\label{finiteinfimizer}}
	\PP \subseteq C + \conv \bigcup_{x \in \bar{S}} F(x) + \cone \bigcup_{x \in \hat{S}} G(x).
\end{equation} 
An element $\bar{x} \in \dom F$ is said to be a {\em minimizer} or a {\em minimizing point} for \eqref{eq:p} if 
$$ \not\exists x \in \R^n:\; F(x) + C \supsetneq F(\bar x)+C.$$
A nonzero element $\hat{x} \in \dom G$ is said to be a {\em minimizer} or a {\em minimizing direction} for \eqref{eq:p} if
$$ \not\exists x \in \R^n:\; G(x)+C \supsetneq G(\hat x)+C.$$
A finite infimizer $(\bar{S},\hat{S})$ is called a {\em solution} to \eqref{eq:p} if all elements of  $\bar{S}$ are minimizing points and all elements of $\hat{S}$ (if any) are minimizing directions.

Existence of solutions to \eqref{eq:p} can be characterized by the homogeneous problem \eqref{eq:q}. 
\begin{prop}[{\cite[Theorem 3.3]{Loe}}]\label{prop22} Let $F:\R^n\rightrightarrows \R^q$ be a polyhedral convex set-valued objective mapping  and let $C \subseteq \R^q$ be a polyhedral convex ordering cone which is regular for $F$. Then \eqref{eq:p} has a solution if and only if \eqref{eq:p} is feasible and $0$ is a minimizer for \eqref{eq:q}, that is
	$$ \not\exists x \in \R^n:\; G(x)+C \supsetneq C.$$
\end{prop}
\begin{proof}
	In \cite[Theorem 3.3]{Loe} it was shown that \eqref{eq:p} has a solution if and only if \eqref{eq:p} is feasible and there is no $x \in \R^n$ with $G(x)+C \supsetneq G(0)+C$. Since $C$ is regular for $F$, we obtain $G(0)+C=C$.
\end{proof}

\section{The natural ordering cone} \label{sec:nat_cone}

Let us assume for the moment that in problem \eqref{eq:p} the upper image $\PP$ is free of lines. Then a polyhedral convex cone $C\subseteq \R^q$ is regular for $F$ if 
$G(0) \subseteq C$.
If an ordering cone $C$ is enlarged to the cone $G(\R^n)+C = \QQ$, the resulting problem \eqref{eq:p} has a solution. 
This can be seen by Proposition \ref{prop22} or directly. In particular, for such a solution $(\bar S,\hat S)$, the direction part is empty, i.e., $\hat S = \emptyset$. Indeed, for $C=\QQ$ we have $G(x) \subseteq C$ for all $x$ which implies that there are no minimizing directions. Thus, enlarging $C$ to $\QQ$ makes the problem {\em bounded}, for a definition and details see e.g.\ \cite{LoeSch13,HeyLoe,Loe}.

For the smallest choice of a regular ordering cone, namely $C=G(0)$, \eqref{eq:p} does not necessarily have a solution.

\begin{ex}\label{ex11}
	Let $F: \R^2 \rightrightarrows \R^2$, where $gr F$ is defined by the inequalities
	$$ x_1 \geq 0,\; x_2 \geq 0,\; y_2 \geq x_2 ,\; x_2 + y_1 \geq 0,\; x_1 + 2x_2 + y_1 \geq y_2.$$
	Since $\gr F$ is a cone, we have $G = F$. Moreover,
	$$
		G(0,0)= \Set{y \in \R^2\given y_2 \geq 0,\; y_1 \geq y_2}.
	$$
	It is a proper subset of 
	$$	G(1,0)= \Set{y \in \R^2 \given y_2 \geq 0,\;  y_1 \geq 0,\; y_1 \geq y_2 - 1}. $$
	We deduce
	$G(1,0)+G(0,0) = G(1,0) \supsetneq G(0,0)$ which means by Proposition \ref{prop22}, that \eqref{eq:p} with ordering cone $C=G(0,0)$ does not have a solution.
	If we choose the regular ordering cone $C=G(\R^2)$, where we have
	$$G(\R^2) = \Set{y \in \R^2\given y_2 \geq 0,\; y_1 \geq -y_2},$$
	we see that $(\Set{0},\emptyset)$ is a solution to \eqref{eq:p}.
\end{ex}

This example shows that the two extremal regular ordering cones $G(0)$ and $G(\R^n)$ can be of limited use. If the ordering cone is too small, a solution may not exist. If it is too big, a solution can provide a limited amount of information. Thus it is natural to ask whether there is a smallest regular ordering cone $C$ such that \eqref{eq:p} still has a solution. We show that such a cone exists: The set 
$$ K \coloneqq  K_F \coloneqq \Set{y \in \R^q \given \exists x \in \R^n:\; y \in G(x),\; 0 \in G(x)}$$
is called the \emph{natural ordering cone} of the set-valued mapping $F$.

\begin{prop}\label{prop4} Let $\gr F \neq \emptyset$. The natural ordering cone $K$ of $F$ is a polyhedral convex cone satisfying
	\begin{equation}\label{eq:kf1}
		 G(0) \subseteq K \subseteq G(\R^n).
	\end{equation}
\end{prop} 
\begin{proof}
	Since $\gr F$ is a convex polyhedron, it has an H-representation 
	$$ \gr F = \Set{(x,y) \in \R^n\times \R^q\given Ax + B y \geq b}.$$
	If the right-hand side in the inequalities is replaced by zero, we obtain an H-representation of the recession cone of $\gr F \neq \emptyset$, i.e., of $\gr G$. Thus we have
	$$ K = \Set{y \in \R^q \given \exists x \in \R^n:\; A x + B y \geq 0,\; A x \geq 0},$$
	which shows that $K$ is a polyhedral convex cone.
	
	To show the first inclusion in \eqref{eq:kf1}, let $y \in G(0)$. Since $G(0)$ is a polyhedral convex cone, we also have $0 \in G(0)$, which implies $y \in K$. To show the second inclusion in \eqref{eq:kf1}, let $y \in K$. By the definition of $K$, there exists $x \in \R^n$ such that $y \in G(x) \subseteq G(\R^n)$. 
\end{proof}

The proposition shows that $K$ is regular for $F$, whenever $\lin K \subseteq G(0)$, in particular, if $K$ is free of lines. 
Both inclusions in \eqref{eq:kf1} can be strict.

\begin{ex}\label{ex12}
	Let $F$ be defined as in Example \ref{ex11}. Then we have
	$$ K = \Set{y \in \R^2\given y_2 \geq 0,\; y_1 \geq 0},$$
	which shows that $G(0) \subsetneq K \subsetneq G(\R^2)$.
\end{ex}

The following result characterizes the existence of solutions of a polyhedral convex set optimization problem \eqref{eq:p} by two conditions. In particular, the result states that $C \supseteq K$ is necessary for the existence of solutions.

\begin{thm}\label{thm2} Let $F:\R^n\rightrightarrows \R^q$ be a polyhedral convex set-valued mapping and let $C \subseteq \R^q$ be a polyhedral convex ordering cone which is regular for $F$. Then  \eqref{eq:p} has a solution if and only if \eqref{eq:p} is feasible,
	\begin{equation}\label{eq:cond}
		-C \cap \QQ \subseteq C
	\end{equation}
	and 
	\begin{equation}\label{eq:cond1}
		C \supseteq K.
	\end{equation}	
\end{thm}
\begin{proof} It is sufficient to consider the case where \eqref{eq:p} is feasible.
	
	 Assume that \eqref{eq:cond} and \eqref{eq:cond1} hold but \eqref{eq:p} does not have a solution. By Proposition \ref{prop22}, there exists some $x \in \R^n$ such that 
	$G(x)+C \supsetneq C$. Since $0 \in C \subseteq G(x)+C$, there exists some $z \in G(x)$ with $z \in -C$. 
	By \eqref{eq:cond} and regularity we obtain $-z \in \lin C \subseteq G(0)$. Thus $0 \in G(x) + \Set{-z} \subseteq G(x) + G(0) = G(x)$. This yields the inclusion $G(x) \subseteq K$. By \eqref{eq:cond1}, $G(x) \subseteq C$ and thus $G(x) + C \subseteq C$, which contradicts the strict converse inclusion above. 
	
Let \eqref{eq:cond1} be violated, i.e., there exists $y \in K$  with $y \not\in C$. By the definition of $K$, there exists $x \in \R^n$ such that $y \in G(x)$ and $0 \in G(x)$. The latter statement yields
$C \subseteq G(x)+C$. Since $y \notin C$ and $y \in G(x) \subseteq G(x) +C$ we obtain the strict inclusion  $C \subsetneq G(x)+C$, which means that \eqref{eq:p} has no solution, by Proposition \ref{prop22}.
	
Let \eqref{eq:cond} be violated, i.e., there exists $z \in -C \setminus C$ such that $z \in G(x)+C$ for some $x \in \R^n$. We deduce $0 \in G(x) + C + C = G(x) + C$ and hence $C \subseteq G(x) + C$.
We have $z \not\in C$ but $z \in G(x)+C$ and so we obtain the strict inclusion $G(x) + C \supsetneq C$. By Proposition \ref{prop22}, this means that \eqref{eq:p} has no solution.
\end{proof}

If the ordering cone $C$ is not regular for $F$, we have similar result, which follows easily from the theorem.

\begin{cor}
	Let $F:\R^n\rightrightarrows \R^q$ be a polyhedral convex set-valued mapping and let $C \subseteq \R^q$ be a polyhedral convex ordering cone. Then  \eqref{eq:p} has a solution if and only if \eqref{eq:p} is feasible,
		\begin{equation}\label{eq:cond_cor}
			-G_C(0) \cap G_C(\R^n) \subseteq G_C(0)
		\end{equation}
		and 
		\begin{equation}\label{eq:cond1_cor}
			G_C(0) \supseteq K_{F_C}.
		\end{equation}		
\end{cor}
\begin{proof}
	The order cone $\bar C \coloneqq G(0) + C$ is regular for the set-valued mapping $\bar F \coloneqq F_C$, see Proposition \ref{prop11}. Replacing $F$ by $\bar F$ and $C$ by $\bar C$ does not change problem \eqref{eq:p}. Finally, the condition \eqref{eq:cond} and \eqref{eq:cond1} applied to $\bar F$ and $\bar C$ lead to the conditions \eqref{eq:cond_cor} and \eqref{eq:cond1_cor}.
\end{proof}

\section{Connection to vector linear programming}\label{sec:vr}

Vector linear programming is a particular case of polyhedral convex set optimization. Consider the vector linear program
	\leqnomode
	\begin{gather}\label{eq:vlp}\tag{VLP}
	  \min\text{$\!_C$} Mx\; \text{ s.t. }\; A x \geq b,
	\end{gather}
	\reqnomode
where $M \in \R^{q\times n}$, $A \in \R^{m \times n}$, $b \in \R^m$ and $C \subseteq \R^q$ being a polyhedral convex ordering cone. 
A tuple $(\bar S,\hat S)$ is called a solution for \eqref{eq:vlp} if $(\bar S,\hat S)$ is a solution of \eqref{eq:p} for the objective mapping
\begin{equation}\label{eq:obj1}
	F(x)\coloneqq \left\{ \begin{array}{cl}
			\Set{Mx} + C &\text{  if } Ax \geq b \\
			\emptyset  &\text{ otherwise }
		\end{array}\right.
\end{equation}
and the ordering cone $C$. For this objective mapping we have 
$$ G(x) = \left\{\begin{array}{cl}
	\{Mx\} + C & \text{ if } A x \geq 0 \\
	\emptyset & \text{ otherwise.} 
\end{array}\right.$$
If $C$ is free of lines, this solution concept coincides with the one in \cite{LoeWei16}. In the general case, it coincides with the solution concept in \cite[Section 4]{Weissing20}. Note that the ordering cone $C$ is regular for the objective mapping $F$ as we have $\lin C \subseteq G(0) = C$. Moreover, we have $\QQ = G(\R^n)$.

\begin{prop}\label{prop8}
	For a feasible vector linear program \eqref{eq:vlp} and its associated polyhedral convex set optimization problem \eqref{eq:p} with objective function of the form \eqref{eq:obj1}, condition \eqref{eq:cond} in Theorem \ref{thm2} implies $K=C$.
\end{prop}
\begin{proof}
	By Proposition \ref{prop4}, we have $G(0) \subseteq K$. The special form of $F$ implies $G(0)=C$. Hence the inclusion $C \subseteq K$ holds. To show the converse inclusion, let $y \in K$. Then there exists $x \in \R^n$ with $Ax \geq 0$ such that $y \in \Set{Mx} + C$ and $0 \in \Set{Mx} + C$. We see that $Mx \in -C \cap \QQ$. Condition \eqref{eq:cond} yields $Mx \in C$. Thus $y \in \Set{Mx} + C \subseteq C+C = C$.  
\end{proof}

We recover a known existence result for vector linear programs, see e.g.\ \cite[Theorem 4.1]{Weissing20} in combination with \cite[Theorem 3.3]{Loe}. The same result was also recovered in \cite{Loe} as a direct consequence of \cite[Theorem 3.3]{Loe}, see Proposition \ref{prop22} above. Here it follows directly from Theorem \ref{thm2}, which is also based on \cite[Theorem 3.3]{Loe}, and Proposition \ref{prop8}.

\begin{cor}\label{cor8}
	The vector linear program \eqref{eq:vlp} has a solution if and only if it is feasible and \eqref{eq:cond} holds.
\end{cor}

The considerations in this section show that the natural ordering cone does not play any role in vector linear programming. It is relevant for more general set optimization problems only. 

In the second part of this section we want to discuss the \emph{vectorial relaxation}, which was introduced in \cite{LoeSch13} for bounded polyhedral convex set optimization problems and was extended in \cite{HeyLoe} to the unbounded case. The idea is to assign a vector linear problem to a given polyhedral convex set optimization problem such that both problems have the same upper image or infimum. This allows to compute a finite infimizer for \eqref{eq:p} by solving a vector linear program, see \cite{LoeSch13, HeyLoe} for details. As pointed out in \cite{LoeSch13}, \eqref{eq:p} and its vectorial relaxation behave differently with respect to minimality. This is the reason why a specific method for computing minimizers has been developed in \cite{LoeSch13} (for bounded problems only).

Let $F:\R^n\rightrightarrows \R^q$ be a polyhedral convex set-valued mapping. We assign to $F$ the mapping
$$ \hat F :\R^n\times \R^q \rightrightarrows \R^q,\quad \hat F(x,u) = \left\{ \begin{array}{cl}
\Set{u} + G(0) & \text{ if } u \in F(x)\\
\emptyset      & \text{ otherwise.}
\end{array}\right.$$
We see that $\hat F$ is of the form \eqref{eq:obj1}, i.e., problem \eqref{eq:p} with objective function $\hat F$ and ordering cone $G(0)$ can be seen as a reformulation of a vector linear program.
The graph of $\hat F$ can be written as
$$ \gr \hat F = \Set{(x,u,y) \in \R^n \times \R^q \times \R^q \given (0,y-u) \in \gr G,\;(x,u) \in \gr F }.$$
Consequently, the recession mapping of $\hat F$ is 
$$ \hat G :\R^n\times \R^q \rightrightarrows \R^q,\quad \hat G(x,u) = \left\{ \begin{array}{cl}
\Set{u} + G(0) & \text{ if } u \in G(x)\\
\emptyset      & \text{ otherwise.}
\end{array}\right.$$

By Proposition \ref{prop8}, we see that $-G(0) \cap \QQ \subseteq G(0)$ implies that the natural ordering cone $\hat K$ of $\hat F$ is just $G(0)$.

The \emph{vectorial relaxation} of the polyhedral convex set optimization problem \eqref{eq:p} is defined as:
\leqnomode
\begin{gather}\label{eq:vr}\tag{VR}
  \min \hat F(x,y) + C\quad \text{ s.t. }\; (x,y) \in \R^n\times \R^q.
\end{gather}
\reqnomode

\begin{prop}\label{prop5} 
	A polyhedral convex cone $C$ is regular for $F$ if and only if it is regular for $\hat F$.
\end{prop}
\begin{proof}
	Follows from $G(0)= \hat G(0,0)$ and $\QQ = \hat \QQ$ and the definition of regularity in \eqref{eq:reg}.
\end{proof}

\begin{cor}\label{cor15} Let $F:\R^n\rightrightarrows \R^q$ be a polyhedral convex set-valued mapping and let $C\subseteq \R^q$ be a  polyhedral convex ordering cone which is regular for $F$. Let \eqref{eq:p} be feasible. Then its vectorial relaxation \eqref{eq:vr} has a solution if and only if $-C\cap \QQ \subseteq C$.
\end{cor}
\begin{proof} This follows from Theorem \ref{thm2} applied to \eqref{eq:vr}. Since $\QQ = \hat \QQ$, condition \eqref{eq:cond} is the same for \eqref{eq:p} and \eqref{eq:vr}. By Proposition \ref{prop8}, \eqref{eq:cond} implies $\hat K = \hat G(0,0)$. Since $C$ is regular for $\hat F$ we deduce $C \supseteq \hat K$, which is condition \eqref{eq:cond1} for \eqref{eq:vr}.
\end{proof}

We close this section by a re-formulation of previous results which points out the difference between a polyhedral convex set optimization problem \eqref{eq:p} and its vectorial relaxation \eqref{eq:vr} in regard to the existence of solutions. The natural ordering cone plays a key role in describing this difference.
 
\begin{cor} \label{cor68} Let $F:\R^n\rightrightarrows \R^q$ be a polyhedral convex set-valued mapping and let $C\subseteq \R^q$ be a  polyhedral convex ordering cone which is regular for $F$.
Then, \eqref{eq:p} has a solution if and only if
\begin{enumerate}[(i)]
	\item the vectorial relaxation \eqref{eq:vr} has a solution and
	\item $C \supseteq K$
\end{enumerate}
\end{cor}
\begin{proof}
	Follows from Theorem \ref{thm2} and Corollary \ref{cor15}
\end{proof}

\section{A modified solution concept}\label{sec:sol}

In Corollary \ref{cor68} we observe that the vectorial relaxation \eqref{eq:vr} can have a solution in case the original problem \eqref{eq:p} does not have a solution. Then the ordering cone $C$ is not a superset of the natural ordering cone $K$. Enlarging the ordering cone to $C+K$ resolves this situation. However, using the original solution concept this has the disadvantage that information of the original problem gets lost.

We introduce in this section a slightly modified solution concept: First we compute a finite infimizer with respect to the original ordering cone $C$. Secondly we require the points and directions of this infimizer to be minimizers with respect to the enlarged ordering cone $C+K$. Third, we indicate minimizing directions the image of which belongs to $K$. 

The \emph{algebraic kernel} of a polyhedral convex set-valued mapping $F:\R^n \rightrightarrows \R^q$ is defined as
$$ \ker F \coloneqq \Set{x \in \R^n \given 0 \in G(x)}.$$
Of course, $\ker F$ is a polyhedral convex cone in $\R^n$.
The natural ordering cone $K$ of $F$ can be expressed as the image of $\ker F$ under $F$'s recession mapping $G$, i.e.,
$$ K = G[\ker F] \coloneqq \Set{G(x)\given x \in \ker F}.$$

As we deal now with different ordering cones, we adapt the notation. A finite infimizer for problem \eqref{eq:p} with objective $F$ and ordering cone $C$ is now called \emph{finite $C$-infimizer of $F$} and a minimizer for \eqref{eq:p} is called \emph{$C$-minimizer of $F$}. Moreover a solution to \eqref{eq:p} is also called $C$-solution.

Consider problem \eqref{eq:p} for a polyhedral convex set-valued objective mapping $F:\R^n\rightrightarrows \R^q$ and a polyhedral convex ordering cone $C \subseteq \R^q$ which is regular for $F$. 

\begin{defn}\label{def:sol}
A triple $(\bar{S},\hat{S},\tilde{S})$ of finite sets $\bar{S} \subseteq \dom F$, $\bar{S} \neq \emptyset$, $\hat{S} \subseteq \dom G \setminus \ker F$ and $\tilde{S} \subseteq \ker F$ is called {\em solution} to \eqref{eq:p} if 
\begin{enumerate}[(i)]
	\item $(\bar S, \hat S \cup \tilde S)$ is a finite $C$-infimizer of $F$, 
	\item the points in $\bar S$ and the directions in $\hat S \cup \tilde S$ are $(C+K)$-minimizers of $F$.
\end{enumerate}
\end{defn}

\begin{prop}
	If $(\bar{S},\hat{S},\tilde{S})$ is a solution to \eqref{eq:p} in the sense of Definition \ref{def:sol}, then $(\bar{S},\hat{S})$ is a ($C+K$)-solution to \eqref{eq:p} in the conventional sense. 
\end{prop}
\begin{proof}
	This is obvious by the definitions.
\end{proof}

The following proposition shows a property of kernel elements, which occur in Definition \ref{def:sol}.

\begin{prop} Let $C$ be a regular ordering cone for $F$.
	Let $x \in \dom F$ and $\tilde x \in \ker F$ such that $G(\tilde x) \not \subseteq C$. Then $F(x + \tilde x ) + C \supsetneq F(x)+C$.
\end{prop}
\begin{proof}
	Since $G$ is the recession mapping of $F$, we have $F(x) + G(\tilde x) \subseteq F(x + \tilde x)$. By the definition of the kernel, $0 \in G(\tilde x)$ and thus the inclusion $F(x) \subseteq F(x) + G(\tilde x)$ holds. Since $C$ is regular for $F$, the recession cone of $F(x)$ belongs to $C$. The assumption $G(\tilde x) \not \subseteq C$ implies the strict inclusion $F(x) + C \subsetneq F(x) + G(\tilde x) + C$. Putting all this together yields the claim $F(x) + C \subsetneq F(x+\tilde x)+C$.
\end{proof}

We illustrate the modified solution concept and the results of this section by two examples. In the first example, the kernel component $\tilde S$ of the solution is empty while in the second example it is not. So the second example shows that the solution concept in Definition \ref{def:sol} can provide additional information, which gets lost by enlarging the ordering cone to $C+K$. 

\begin{ex}
	Let $F$ be defined as in Example \ref{ex11}. Then, the triple 
	$$(\bar{S},\hat{S},\tilde{S}) = (\Set{(0,0)^T},\Set{(0,1)^T},\emptyset)$$
	 is a solution in the sense of Definition \ref{def:sol}. As shown above, the condition $C\supseteq K$ is violated here. Nevertheless, this solution does not contain an element of $\ker F$. Since $(1,0)^T$ belongs to $\ker F$ and $G(1,0) \not\subseteq C$, we have
	$F(r,t)+C \supsetneq F(s,t)+ C$ whenever $r > s \geq 0$ and $t \geq 0$.
\end{ex}

\begin{ex} \cite[Example 17]{HeyLoe}
		Let $C=\R^2_+$ and let $F: \R^2 \rightrightarrows \R^2$ be defined by 
			$$ \gr F = \cone \left\{ \begin{pmatrix}\phantom{-}1 \\\phantom{-}0 \\ \phantom{-}2 \\-1 \end{pmatrix},\;
			                          \begin{pmatrix} 1 \\ 0 \\0 \\ 0 \end{pmatrix},\;
									   \begin{pmatrix} \phantom{-}0 \\  \phantom{-}1 \\ -1 \\  \phantom{-}2 \end{pmatrix},\;
										\begin{pmatrix} 0 \\ 0 \\1 \\ 0 \end{pmatrix},\;
									     \begin{pmatrix} 0 \\ 0 \\0 \\ 1 \end{pmatrix}\right\}.$$
		Since $\gr F$ is a cone, we have $F=G$. An H-representation of $\gr F$ is given by the inequalities
		$$ x_1 \geq 0,\; x_2 \geq 0,\;  x_2 + y_1 \geq 0,\; -3 x_2 + y_1 + 2 y_2 \geq 0,\; x_1 - 2 x_2 + y_2 \geq 0.$$
		We have 
		\begin{align*}
			G(\R^2) = \QQ &= \Set{y \in \R^2 \given y_1 + 2 y_2 \geq 0,\; 2 y_1 + y_2 \geq 0},\\
			K &= \Set{y \in \R^2 \given y_1 \geq 0,\; 2 y_1 + y_2 \geq 0},\\
			C &= G(0,0) = \R^2_+,			
		\end{align*}
		in particular, $C \subsetneq K \subsetneq \QQ$.
		Then, the triple 
			$$(\bar{S},\hat{S},\tilde{S}) = (\Set{(0,0)^T},\Set{(0,1)^T},\Set{(1,0)^T})$$
		is a solution in the sense of Definition \ref{def:sol}. It has a nonempty third component, which is not redundant. 
		Since $(1,0)^T$ belongs to $\ker F$ and $G(1,0) \not\subseteq C$, we have
	$F(r,t)+C \supsetneq F(s,t)+ C$ whenever $r > s \geq 0$ and $t \geq 0$.
\end{ex}

%\bibliography{ref}

\end{document}